\documentclass[11 pt]{article}
\usepackage{amsmath}
\usepackage{amsthm}
\usepackage{graphicx}
\usepackage{bbm,amssymb}

\usepackage[hyperfootnotes=false,colorlinks]{hyperref}

\def\proofof#1{{ \medbreak \noindent {\bf Proof of #1.} }}

\newcommand{\floor}[1]{\lfloor {#1} \rfloor}

\newcommand{\de}{\partial}
\newcommand{\old}[1]{}
\newcommand{\set}[1]{\left\{#1\right\}}

\newcommand{\abs}[1]{\left\vert#1\right\vert}
\newcommand{\vol}{\rm{vol}\,}
\newcommand{\e}{\epsilon}
\newcommand{\arxiv}[1]{{\tt \href{http://arxiv.org/abs/#1}{arXiv:#1}}}

\newtheorem{theorem}{Theorem}[section]
\newtheorem{prop}[theorem]{Proposition}
\newtheorem{lemma}[theorem]{Lemma}

\theoremstyle{remark}
\newtheorem*{remark}{Remark}

\theoremstyle{definition}

\def\Var{{\mathrm{Var}}}
\def\Cov{{\mathrm{Cov}}}

\def\Grid{\mathcal{G}}

\def \eps {\varepsilon}
\def\Re{\mathrm{Re}}

\def\basis{\mathbf{e}}
\def\B{\mathbf{B}} 

\def\f{\varphi}

\title{Internal DLA and the Gaussian free field}
\author{David Jerison \and Lionel Levine\footnote{Supported by an NSF Postdoctoral Research Fellowship.} \and Scott Sheffield\footnote{Partially supported by NSF grant
DMS-0645585.}}
\date{January 31, 2011}

\DeclareSymbolFont{AMSb}{U}{msb}{m}{n}
\DeclareMathSymbol{\C}{\mathbin}{AMSb}{"43}
\DeclareMathSymbol{\EE}{\mathbin}{AMSb}{"45}
\DeclareMathSymbol{\N}{\mathbin}{AMSb}{"4E}
\DeclareMathSymbol{\PP}{\mathbin}{AMSb}{"50}
\DeclareMathSymbol{\Q}{\mathbin}{AMSb}{"51}
\DeclareMathSymbol{\R}{\mathbin}{AMSb}{"52}
\DeclareMathSymbol{\Z}{\mathbin}{AMSb}{"5A}

\begin{document}

\maketitle
\renewcommand{\thefootnote}{}
\renewcommand{\thefootnote}{\arabic{footnote}}

\begin{abstract}
In previous works, we showed that the internal DLA cluster on $\Z^d$ with $t$ particles is almost surely spherical up to a maximal error of $O(\log t)$ if $d=2$ and $O(\sqrt{\log t})$ if $d \geq 3$.  This paper addresses ``average error'': in a certain sense, the average deviation of internal DLA from its mean shape is of {\em constant} order when $d=2$ and of order $r^{1-d/2}$ (for a radius $r$ cluster) in general.  Appropriately normalized, the fluctuations (taken over time and space) scale to a variant of the Gaussian free field.
\end{abstract}

\pagebreak
\tableofcontents

\section{Introduction}

\subsection{Overview}
We study scaling limits of \emph{internal diffusion limited aggregation} (``internal DLA''), a growth model introduced in \cite{MD,DF}.
In internal DLA, one inductively constructs an {\bf occupied set} $A_t \subset \Z^d$ for each time $t \geq 0$ as follows: begin with $A_0 = \emptyset$ and $A_1 = \{0\}$, and let $A_{t+1}$ be the union of $A_t$ and the first place a random walk from the origin hits $\Z^d \setminus A_t$.

The purpose of this paper is to study the growing family of sets $A_t$.
Following the pioneering work of \cite{LBG}, it is by now well known that, for large $t$, the set $A_t$ approximates an origin-centered Euclidean lattice ball $\B_r:= B_r(0) \cap \Z^d$ (where $r=r(t)$ is such that $B_r(0)$ has volume $t$).  The authors recently showed that this is true in a fairly strong sense \cite{JLS0, JLS, JLS2}: the maximal distance from a point where $1_{A_t} - 1_{\B_r}$ is non-zero to $\partial B_r(0)$ is a.s.\ $O(\log t)$ if $d=2$ and $O(\sqrt{\log t})$ if $d \geq 3$.  In fact, if $C$ is large enough, the probability that this maximal distance exceeds $C\log t$ (or $C \sqrt{\log t}$ when $d \geq 3$) decays faster than any fixed (negative) power of $t$.  Some of these results are obtained by different methods in \cite{AG,AG2}.

This paper will ask what happens if, instead of considering the maximal distance from $\partial B_r(0)$ at time $t$, we consider the ``average error'' at time $t$ (allowing inner and outer errors to cancel each other out).  It turns out that in a distributional ``average fluctuation'' sense, the set $A_t$ deviates from $B_r(0)$ by only a constant number of lattice spaces when $d = 2$ and by an even smaller amount when $d \geq 3$.  Appropriately normalized, the fluctuations of $A_t$, taken over time and space, define a distribution on $\R^d$ that converges in law to a variant of the Gaussian free field (GFF): a random distribution on $\R^d$ that we will call the {\bf augmented Gaussian free field}.  (It can be constructed by defining the GFF in spherical coordinates and replacing variances associated to spherical harmonics of degree $\ell$ by variances associated to spherical harmonics of degree $\ell + 1$; see \textsection\ref{ss.augmentedGFF}.)  The ``augmentation'' appears to be related to a damping effect produced by the mean curvature of the sphere (as discussed below).\footnote{Consider continuous time internal DLA on the half cylinder $(\Z/m\Z)^{d-1} \times \Z_+$, with particles started uniformly on $(\Z/m\Z)^{d-1} \times \{0\}$.  Though we do not prove this here, we expect the cluster boundaries to be approximately flat cross-sections of the cylinder, and we expect the fluctuations to scale to the {\em ordinary} GFF on the half cylinder as $m \to \infty$.}

To our knowledge, no central limit theorem of this kind has been previously conjectured in either the physics or the mathematics literature.
The appearance of the GFF and its ``augmented'' variants is a particular surprise.  (It implies that internal DLA fluctuations --- although very small --- have long-range correlations and that, up to the curvature-related augmentation, the fluctuations in the direction transverse to the boundary of the cluster are of a similar nature to those in the tangential directions.)
  Nonetheless, the heuristic idea is easy to explain.  Before we state the central limit theorems precisely (\textsection\ref{ss.twostatement} and \textsection\ref{ss.generalstatement}), let us explain the intuition behind them.

Write a point $x \in \mathbb R^d$ in polar coordinates as $r \theta$ for $r \geq 0$ and $\theta$ on the unit sphere.  Suppose that at each time $t$ the boundary of $A_t$ is approximately parameterized by $r_t(\theta) \theta$ for a function $r_t$ defined on the unit sphere.  Write
	\[ r_t(\theta) = (t/\omega_d)^{1/d} + \rho_t(\theta) \]
where $\omega_d$ is the volume of the unit ball in $\R^d$.  The $\rho_t(\theta)$ term measures the deviation from circularity  of the cluster $A_t$ in the direction $\theta$.  How do we expect $\rho_t$ to evolve in time?  To a first approximation, the angle at which a random walk exits $A_t$ is a uniform point on the unit sphere.  If we run many such random walks, we obtain a sort of Poisson point process on the sphere, which has a scaling limit given by space-time white noise on the sphere.  However, there is a smoothing effect (familiar to those who have studied the continuum analog of internal DLA: the famous Hele-Shaw model for fluid insertion, see the reference text \cite{GV}) coming from the fact that places where $\rho_t$ is small are more likely to be hit by the random walks, hence more likely to grow in time.  There is also secondary damping effect coming from the mean curvature of the sphere, which implies that even if (after a certain time) particles began to hit all angles with equal probability, the magnitude of $\rho_t$ would shrink as $t$ increased and the existing fluctuations were averaged over larger spheres.

The white noise should correspond to adding independent Brownian noise terms to the spherical Fourier modes of $\rho_t$.  The rate of smoothing/damping in time should be approximately given by $\Lambda \rho_t$ for some linear operator $\Lambda$ mapping the space of functions on the unit sphere to itself.  Since the random walks approximate Brownian motion (which is rotationally invariant), we would expect $\Lambda$ to commute with orthogonal rotations, and hence have spherical harmonics as eigenfunctions.  With the right normalization and parameterization, it is therefore natural to expect the spherical Fourier modes of $\rho_t$ to evolve as independent Brownian motions subject to linear ``restoration forces'' (a.k.a.\ Ornstein-Uhlenbeck processes) where the magnitude of the restoration force depends on the degree of the corresponding spherical harmonic.  It turns out that the restriction of the (ordinary or augmented) GFF on $\R^d$ to a centered volume $t$ sphere evolves in time $t$ in a similar way.

Of course, as stated above, the ``spherical Fourier modes of $\rho_t$'' have not really been defined (since the boundary of $A_t$ is complicated and generally {\em cannot} be parameterized by $r_t(\theta)\theta$).  In the coming sections, we will define related quantities that (in some sense) encode these spherical Fourier modes and are easy to work with.  These quantities are the martingales obtained by summing discrete harmonic polynomials over the cluster $A_t$.

The heuristic just described provides intuitive interpretations of the results given below.  Theorem~\ref{t.fluctuations}, for instance, identifies the weak limit as $t \to \infty$ of the internal DLA fluctuations from circularity at a fixed time $t$: the limit is the two-dimensional augmented Gaussian free field restricted to the unit circle $\partial B_1(0)$, which can be interpreted in a distributional sense as the random Fourier series
\begin{equation}
\label{eq:fourier}
\frac{1}{\sqrt{2\pi}}\left[\alpha_0/\sqrt2 + \sum_{k=1}^\infty  \alpha_k
\frac{\cos k\theta} {\sqrt{k+1}}
+
\beta_k\frac{\sin k\theta} {\sqrt{k+1}} \right]
\end{equation}
where $\alpha_k$ for $k\geq 0$ and $\beta_k$ for $k\geq 1$ are independent standard Gaussians.  The ordinary two-dimensional GFF restricted to the unit circle is similar, except that $\sqrt{k +1 }$ is replaced by $\sqrt k$.

The series \eqref{eq:fourier} --- unlike its counterpart for the one-dimensional Gaussian free field, which is a variant of Brownian bridge --- is a.s.\ divergent, which is why we use the dual formulation explained in \textsection\ref{ss.generalstatement}.  The dual formulation of \eqref{eq:fourier} amounts to a central limit theorem, saying that for each $k\geq 1$ the real and imaginary parts of
 	\[ M_k = \frac1r \sum_{z \in A_{\pi r^2}} \left( \frac{z}{r}\right)^k \]
converge in law as $r \to \infty$ to normal random variables with variance $\frac{\pi}{2(k+1)}$ (and that $M_j$ and $M_k$ are asymptotically uncorrelated for $j\neq k$).  See \cite[\textsection6.2]{FL} for numerical data on the moments $M_k$ in large simulations.

\subsection{FKG inequality statement and continuous time}

Before we set about formulating our central limit theorems precisely, we mention a previously overlooked fact.  Suppose that we run internal DLA in continuous time by adding particles at Poisson random times instead of at integer times: this process we will denote by $A_{T(t)}$ (or often just $A_T$) where $T(t)$ is the counting function for a Poisson point process in the interval $[0,t]$ (so $T(t)$ is Poisson distributed with mean~$t$).  We then view the entire history of the IDLA growth process as a (random) function on $[0,\infty) \times \Z^d$, which takes the value $1$ or $0$ on the pair $(t,x)$ accordingly as $x \in A_{T(t)}$ or $x \notin A_{T(t)}$.  Write $\Omega$ for the set of functions $f: [0,\infty) \times \Z^d \to \{0,1\}$ such that $f(t,x) \leq f(t',x)$ whenever $t \leq t'$, endowed with the coordinate-wise partial ordering.  Let $\PP$ be the distribution of $\{A_{T(t)}\}_{t\geq 0}$, viewed as a probability measure on $\Omega$.

\begin{theorem}
\label{t.fkg}
{\em (FKG inequality)}
For any two increasing functions $F,G \in L^2(\Omega,\PP)$, the random variables $F( \{A_{T(t)} \}_{t \geq0} )$ and $G ( \{A_{T(t)} \}_{t \geq 0} )$ are nonnegatively correlated.
\end{theorem}

One example of an increasing function is the total number $\# A_{T(t)} \cap X$ of occupied sites in a fixed subset $X \subset \Z^d$ at a fixed time $t$.  One example of a decreasing function is the smallest $t$ for which all of the points in $X$ are occupied.  Intuitively, Theorem \ref{t.fkg} means that on the event that one point is absorbed at a late time, it is conditionally more likely for all other points to be absorbed late. The FKG inequality is an important feature of the discrete and continuous Gaussian free fields \cite{Sheffield}, so it is interesting (and reassuring) that it appears in internal DLA at the discrete level.

Note that sampling a continuous time internal DLA cluster at time $t$ is equivalent to first sampling a Poisson random variable $T$ with expectation $t$ and then sampling an ordinary internal DLA cluster of size~$T$.  (By the central limit theorem, $|t-T|$ has order $\sqrt t$ with high probability.)  Although using continuous time amounts to only a modest time reparameterization (chosen independently of everything else) it is aesthetically natural.  Our use of ``white noise'' in the heuristic of the previous section implicitly assumed continuous time.  (Otherwise the total integral of $\rho_t$ would be deterministic, so the noise would have to be conditioned to have mean zero at each time.)

\subsection{Main results in dimension two} \label{ss.twostatement}

\begin{figure}[htbp]
\begin{center}
\begin{tabular}{cc}
\includegraphics[width=.49\textwidth]{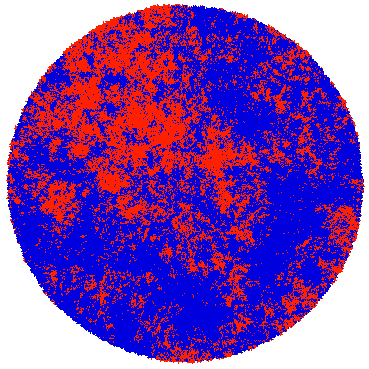} &
\includegraphics[width=.49\textwidth]{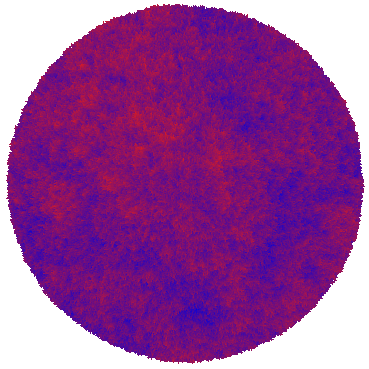} \\
(a) & (b)
\end{tabular}
\caption{\label{fig:shadedcluster} (a) Continuous-time IDLA cluster $A_{T(t)}$ for $t=10^5$.
Early points (where $L$ is negative) are colored red, and late points (where $L$ is positive) are colored blue. (b) The same cluster, with the function $L(x)$ represented by red-blue shading.}
\end{center}
\end{figure}

For $x \in \Z^2$ write $$F(x) : = \inf \{t: x \in A_{T(t)} \}$$ and $$L(x) := \sqrt{F(x)/\pi} - |x|.$$  In words, $L(x)$ is the difference between the radius of the area $t$ disk --- at the time $t$ that $x$ was absorbed into $A_T$ --- and $|x|$.  It is a measure of how much later or earlier $x$ was absorbed into $A_T$ than it would have been if the sets $A_{T(t)}$ were {\em exactly} centered discs of area $t$.  By the main result of \cite{JLS}, almost surely
	\[ \limsup_{x \in \Z^2} \frac{L(x)}{\log |x|} < \infty. \]

The coloring in Figure~\ref{fig:shadedcluster}(a) indicates the sign of the function $L(x)$, while Figure~\ref{fig:shadedcluster}(b) illustrates the magnitude of $L(x)$ by shading.  Note that the use of continuous time means that the average of $L(x)$ over $x$ may differ substantially from~$0$.  Indeed we see that --- in contrast with the corresponding discrete-time figure of \cite{JLS} ---  there are noticeably fewer early points than late points in Figure~\ref{fig:shadedcluster}(a), which corresponds to the fact that in this particular simulation $T(t)$ was smaller than $t$ for most values of $t$.  Since for each fixed $x \in \Z^2$ the quantity $L(x)$ is a decreasing function of $A_t(x)$, the FKG inequality holds for $L$ as well.  The positive correlation between values of $L$ at nearby points is readily apparent from the figure.

Identify $\R^2$ with $\C$ and let $H_0$ be the linear span of the set of functions on $\C$ of the form $\Re (a z^k) f(|z|)$ for $a \in \C$, $k \in \Z_{\geq 0}$, and $f$ smooth and compactly supported on $\R_{>0}$.  The space $H_0$ is obviously dense in $L^2(\C)$, and it turns out to be a convenient space of test functions.  The augmented GFF (and its restriction to $\partial B_1(0)$) will be defined precisely in \textsection\ref{ss.generalstatement} and \textsection\ref{ss.augmentedGFF}.

\begin{theorem} \label{t.lateness} {\em (Weak convergence of the lateness function)} As $R \to \infty$, the rescaled functions on $\R^2$ defined by $G_R((x_1,x_2)) := L((\floor{R x_1}, \floor{R x_2}))$ converge to the augmented Gaussian free field $h$ in the following sense: for each set of test functions $\phi_1, \ldots, \phi_k$ in $H_0$, the joint law of the inner products $(\phi_j, G_R)$ converges to the joint law of $(\phi_j, h)$.
\end{theorem}

\begin{figure}[htbp]
\begin{center}
\includegraphics[width=.4\textwidth]{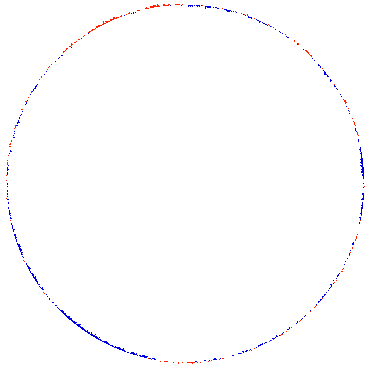} \\ ~ \\
\includegraphics[width=\textwidth]{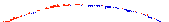}
\caption{\label{fig:symmetricdifference} Top: Symmetric difference between IDLA cluster $A_{T(t)}$ at continuous time $t=10^5$ and the disk of radius $\sqrt{t/\pi}$.  Bottom: closeup of a portion of the boundary.  Sites outside the disk are colored red if they belong to $A_{T(t)}$; sites inside the disk are colored blue if they do not belong to $A_{T(t)}$.}
\end{center}
\end{figure}

Our next result addresses the fluctuations from circularity at a fixed time, as illustrated in Figure~\ref{fig:symmetricdifference}.

\begin{theorem} \label{t.fluctuations}
{\em (Fluctuations from circularity)}
Consider the distribution with point masses on $\R^2$ given by
\begin{equation}\label{e.et} E_t := r^{-1} \sum_{x \in \mathbb Z^2} \bigl( 1_{x \in A_{T(t)}} - 1_{x \in \B_r} \bigr) \delta_{x/r},\end{equation}
where $r = \sqrt{t/\pi}$.  As $t \to \infty$, the $E_t$ converge to the restriction of the augmented GFF to $\partial B_1(0)$, in the sense that for each set of test functions $\phi_1, \ldots, \phi_k$ in $H_0$, the joint law of $(\phi_j, E_t)$ converges to the joint law of $\Phi_h(\phi_j,\pi)$ (a Gaussian process defined in \textsection\ref{ss.generalstatement}).
\end{theorem}

\subsection{Main results in general dimensions} \label{ss.generalstatement}

In this section, we will extend Theorem \ref{t.fluctuations} to general dimensions and to a range of times (instead of a single time).  That is, we will try to understand scaling limits of the discrepancies of the sort depicted in Figure \ref{fig:symmetricdifference} (interpreted in some sense as random distributions) in general dimensions and taken over a range of times.  However, some caution is in order. By classical results in number theory (see the survey \cite{IKKN} for their history), the size of $\B_r = B_r(0) \cap \Z^d$ is approximately the volume of $B_r(0)$ --- but with errors of order $r^{d-2}$ (i.e., both $O(r^{d-2})$ and $\Omega(r^{d-2})$) in all dimensions $d \geq 5$.  The errors in dimension $d=4$ are of order $r^{d-2}$ times logarithmic factors that grow to infinity.  It remains a famous open problem in number theory to estimate the errors when $d \in \{2,3\}$.  (When $d=2$ this is called Gauss's circle problem.)

These number theoretic results imply that $\# \B_{r(t)}$ is, as a function of $t$, much more irregular than the size $T(t)$ of the cluster obtained in continuous time internal DLA, at least when $d \geq 5$.  The results also imply that even if points were added to $A_t$ precisely in order of increasing radius, the difference between the radius of $A_t$ and the radius of $B_{r(t)}(0)$ would fail to be $o(r^{-1})$ when $d\geq 5$ and fail to be $O(r^{-1})$ when $d = 4$.

On the other hand, we will see that the kinds of fluctuations that emerge from internal DLA randomness are of the order that one would obtain by spreading an extra $r^{d/2} \sim \sqrt t$ particles over a constant fraction of the spherical boundary, which is also what one obtains by changing the radius (along some or all of the boundary) by $r^{1-d/2}$.  This implies that the higher dimensional analog of Theorem \ref{t.fluctuations} cannot be true exactly the way it is stated if $d \geq 4$.  Indeed, suppose that we define $E_t$ analogously to \eqref{e.et} as
\[ E_t = r^{-d/2} \sum_{x \in \mathbb Z^d} \bigl( 1_{x \in A_{T(t)}} - 1_{x \in \B_r} \bigr) \delta_{x/r}, \]
and let $\phi$ be a test function that is equal to $1$ in a neighborhood of $\partial B_1(0)$.  Then the results mentioned above imply that \[ (E_t, \phi) = r^{-d/2} \bigl(T(t) -  \# \B_r \bigr) \] cannot converge in law to a finite random variable if $d \geq 4$.

It is therefore a challenge to formulate a central limit theorem for the (small) fluctuations of internal DLA that is not swamped by these (potentially large) number theoretic irregularities.  We will see below that this can be achieved by replacing $\B_r$ with different ball approximations (the so-called ``divisible sandpiles'') that are in some sense even ``rounder'' than the lattice balls themselves.  We will also have to define and interpret the (augmented) GFF in a particular way.

Given smooth real-valued functions $f$ and $g$ on $\R^d$, write
\begin{equation*} \label{e.dirichletinnerproduct} (f,g)_{\nabla} = \int_{\mathbb \R^d} \nabla f(x) \cdot \nabla g(x)dx.\end{equation*}
Here and below $dx$ denotes Lebesgue measure on $\R^d$.
Given a bounded domain $D \subset \R^d$, let $H(D)$ be the Hilbert space closure in $(\cdot,\cdot)_{\nabla}$ of the set of smooth compactly supported functions on $D$.  We define $H = H(\R^d)$ analogously except that the functions are taken modulo additive constants.  The Gaussian free field (GFF) is defined formally by \begin{equation} \label{e.hsum} h:=\sum_{i=1}^\infty \alpha_i f_i,\end{equation}
where the $f_i$ are any fixed $(\cdot, \cdot)_\nabla$ orthonormal basis for $H$ and the $\alpha_i$ are i.i.d.\ mean zero, unit variance normal random variables.  (One also defines the GFF on $D$ similarly, using $H(D)$ in place of $H$.)  The augmented GFF will be defined similarly below, but with a slightly different inner product.

Since the sum \eqref{e.hsum} a.s.\ does not converge within $H$, one has to think a bit about how $h$ is defined.  Note that for any {\em fixed} $f = \sum \beta_i f_i \in H$, the quantity $(h,f)_\nabla := \sum (\alpha_i f_i, f)_{\nabla} = \sum \alpha_i \beta_i$ is almost surely finite and has the law of a centered Gaussian with variance $\|f\|_\nabla = \sum |\beta_i|^2$.  However, there a.s.\ exist some functions $f \in H$ for which the sum does not converge, and $(h, \cdot)_\nabla$ cannot be considered as a continuous functional on all of $H$.  Rather than try to define $(h,f)_\nabla$ for all $f \in H$, it is often more convenient and natural to focus on some subset of $f$ values (with dense span) on which $f \mapsto (h, f)_\nabla$ is a.s.\ a continuous function (in some topology).  Here are some sample approaches to defining a GFF on $D$:

\begin{enumerate}
\item {\bf $h$ as a random distribution:} For each smooth, compactly supported $\phi$, write $(h,\phi) := (h, -\Delta^{-1} \phi)_\nabla$, which (by integration by parts) is formally the same as $\int h(x) \phi(x) dx$.  This is almost surely well defined for all such $\phi$ and makes $h$ a random distribution \cite{Sheffield}.  (If $D = \R^d$ and $d=2$, one requires $\int \phi(x)dx = 0$, so that $(h,\phi)$ is defined independently of the additive constant.  When $d>2$ one may fix the additive constant by requiring that the mean of $h$ on $B_r(0)$ tends to zero as $r \to \infty$ \cite{Sheffield}.)
\item {\bf $h$ as a random continuous $(d+1)$-real-parameter function:} For each $\eps > 0$ and $x \in \R^d$, let $h_\eps(x)$ denote the mean value of $h$ on $\partial B_\eps(x)$.  For each fixed $x$, this $h_\eps(x)$ is a Brownian motion in time parameterized by $-\log \eps$ in dimension $2$, or $-\eps^{2-d}$ in higher dimensions \cite{Sheffield}.  For each fixed $\eps$, the function $h_\eps$ can be thought of as a regularization of $h$ (a point of view used extensively in \cite{DS}).
\item {\bf $h$ as a family of ``distributions'' on origin-centered spheres:}  For each polynomial function $\psi$ on $\R^d$ and each time $t$, define $\Phi_h(\psi,t)$ to be the integral of $h \psi$ over $\partial B_r(0)$ where $B_r(0)$ is the origin-centered ball of volume $t$.  We actually lose no generality in requiring $\psi$ to be a harmonic polynomial on $\R^d$, since the restriction of any polynomial to $\partial B_r(0)$ agrees with the restriction of a (unique) harmonic polynomial.
\end{enumerate}

The last approach turns out to be particularly natural for our purposes.  Using this approach, we will now give our first definition of the augmented GFF: it is the centered Gaussian function $\Phi_h$ for which
\begin{equation} \label{e.cov} \Cov\bigl(\Phi_h(\psi_1, t_1), \Phi_h(\psi_2,t_2)\bigr) = \int_{B_r(0)} \psi_1(x)\psi_2(x)dx,\end{equation}
for all harmonic polynomials $\psi_1$ and $\psi_2$, where $B_r(0)$ is the ball of volume $\min \{t_1, t_2 \}$.
In particular, taking $\psi_1 = \psi_2 = \psi$, we find that
\begin{equation} \label{e.var} \Var\bigl(\Phi_h(\psi, t)\bigr) = \int_{B_r(0)} \psi(x)^2 dx.\end{equation}

Though not immediately obvious from the above, we will see in \textsection\ref{ss.augmentedGFF} that this definition is very close to that of the ordinary GFF.  Now, for each integer~$m$ and harmonic polynomial~$\psi$, there is a discrete harmonic polynomial $\psi_{(m)}$ on $\frac{1}{m} \mathbb Z^d$ (defined precisely in \textsection\ref{s.polynomials}) that approximates $\psi$ in the sense that
$\psi - \psi_{(m)}$ is a polynomial of degree at most $k-2$, where $k$ is the degree of $\psi$.
In particular, if we fix $\psi$ and limit our attention to $x$ in a fixed bounded subset of $\R^d$, then we have $|\psi_{(m)}(x)-\psi(x)| = O(1/m^2)$.
Let $\Grid$ denote the grid comprised of the edges connecting nearest neighbor vertices of $\mathbb Z^d$.   (As a set, $\Grid$ consists of the points in $\mathbb R^d$ with at most one non-integer coordinate.)    As in \cite{JLS}, we extend the definition of $\psi_{(m)}$ to $\Grid$ by linear interpolation.

Now write
	\begin{align} \label{e.phia} \Phi_A^m(\psi,t) :=
	m^{-d/2}\sum_{x \in A_{T(m^dt)}} \psi_{(m)}(x/m) - t \psi_{(m)}(0).
	\end{align}
This random variable measures to what extent the mean value property for the discrete harmonic polynomial $\psi_{(m)}$ fails for the set $A_{T(m^dt)}$.  As such, it is a way of measuring the deviation of $A_{T(m^d t)}$ from circularity.

\begin{theorem} \label{t.highdconvergence}
Let $h$ be the augmented GFF, and $\Phi_h$ as discussed above.  Then as $m \to \infty$, the random functions $\Phi_A^m$ converge in law to $\Phi_h$ (w.r.t.\, the smallest topology that makes $\Phi \mapsto \Phi(\psi,t)$ continuous for each $\psi$ and $t$).  In other words, for any finite collection of pairs $(\psi_1, t_1), \ldots, (\psi_k,t_k)$, the joint law of the $\Phi_A^m(\psi_i,t_i)$ converges to the joint law of the $\Phi_h(\psi_i,t_i)$.
\end{theorem}

\begin{remark}
Theorem \ref{t.highdconvergence} does not really address the discrepancies between $\B_r$ and $A_T$ (which, as we remarked, can be very large, in particular in the case that $\psi$ is a constant function).  However, it can be interpreted as a measure of the discrepancy between $A_T$ and the so-called \emph{divisible sandpile}, which is a function $w_t:\Z^d \to [0,1]$ defined for all $t \geq 0$.  The quantity $w_t(x)$ represents the amount of mass that ends up at $x$ if one begins with $t$ units of mass at the origin and then ``spreads'' the mass around according to certain rules that ensure that the final amount of mass at each site is at most one.  For fixed $x$, the quantity $w_t(x)$ is a continuously increasing function of $t$, and moreover there exists a constant $c$ depending only on the dimension~$d$, such that $w_t(x) = 1$ if $|x| < r(t) -c$ and $w_t(x) = 0$ if $|x| > r(t) + c$ \cite{LP}.  An important property of $w_t(x)$ is that for any discrete harmonic function $\psi$ on $\Z^d$ we have $\sum_{x \in \Z^d} w_t(x) \psi(x) = 0$.  It is natural to replace \eqref{e.et} with
\begin{equation} \label{e.et2} \tilde E_t := r^{-d/2} \sum_{x \in \mathbb Z^d} \bigl( 1_{x \in A_{T(t)}} - w_t(x) \bigr) \delta_{x/r},\end{equation}
and interpret Theorem \ref{t.highdconvergence} as a statement about the distributional limit of $\tilde E_t$.
\end{remark}

\begin{remark}
Even with the replacement above, Theorem \ref{t.highdconvergence} differs from Theorem \ref{t.fluctuations}, since it requires that we use only harmonic polynomial test functions $\psi$ and also requires that we replace them with approximations $\psi_{(m)}$ on the discrete level.
It is natural to ask, in general dimensions, what happens when we try to modify the statement of Theorem \ref{t.highdconvergence} (interpreted as a sort of distributional limit statement for \eqref{e.et2}) to make it read like the distributional convergence statement of Theorem \ref{t.fluctuations}.  We will discuss this in more detail in \textsection\ref{ss.fixedtimeproof}, but we can summarize the situation roughly as follows:

\begin{table}[htbp]
\begin{center}
\begin{tabular}{ll}
\textbf{Modification}     &                           \textbf{When it matters} \\[3pt]
\hline
\begin{tabular}{ll}
\textbf{Replacing} $w_t$ in \eqref{e.et2} with $1_{\B_r}$ \\ ~
\end{tabular}
&
\begin{tabular}{ll} No effect when $d=2$.  \\ Invalidates result when $d > 3$.
\end{tabular}
\\[3pt]
\hline
\begin{tabular}{ll}
\textbf{Keeping} $w_t$ in \eqref{e.et2} but \\ \textbf{Replacing} $\psi_{(m)}$ with $\psi$
\end{tabular}
 &
 \begin{tabular}{ll}
No effect if $d \in \{2,3,4,5 \}$. \\  Unclear if $d > 5$.
 \end{tabular}
 \\[3pt]
\hline
\begin{tabular}{ll}
\textbf{Keeping} $w_t$ in \eqref{e.et2} but \\ \textbf{Replacing} $\psi_{(m)}$ with general \\ smooth test function $\phi$.
\end{tabular}
&
\begin{tabular}{ll}
No effect if $d \in \{2,3\}$.\\ Probably invalidates result if $d>3$.
\end{tabular}
\end{tabular}
\end{center}
\end{table}

The restriction to harmonic $\psi$ (as opposed to a more general test function $\phi$) seems to be necessary in large dimensions because otherwise the derivative of the test function along $\partial B_1(0)$ appears to have a non-trivial effect on \eqref{e.et2} (see \textsection\ref{ss.fixedtimeproof}).  This is because \eqref{e.et2} has a lot of positive mass just outside of the unit sphere and a lot of negative mass just inside the unit sphere.  It may be possible to formulate a version of Theorem \ref{t.highdconvergence} (involving some modification of the ``mean shape'' described by $w_t$) that uses test functions that depend only on $\theta$ in a neighborhood of the sphere (instead of using only harmonic test functions), but we will not address this point here.  Deciding whether Theorem \ref{t.lateness} as stated extends to higher dimensions requires some number theoretic understanding of the extent to which the discrepancies between $w_t$ and $1_{\B_r}$ (as well as the errors that come from replacing a $\psi_{(m)}$ with a smooth test function $\phi$) average out when one integrates over a range of times.  We will not address these points here either.
\end{remark}

\subsection{Comparing the GFF and the augmented GFF} \label{ss.augmentedGFF}

We may write a general vector in $\R^d$ as $r \theta$ where $r \in [0, \infty)$ and $\theta \in S^{d-1} := \partial B_1(0)$.  We write the Laplacian in spherical coordinates as \begin{equation} \label{e.sphericallaplace} \Delta = r^{1-d} \frac{\partial}{\partial r}r^{d-1} \frac{\partial}{\partial r} + r^{-2} \Delta_{S^{d-1}}.\end{equation}
A polynomial $f \in \R[x_1,\ldots,x_d]$ is called \emph{harmonic} if $\Delta f$ is the zero polynomial.
Let $V_\ell$ denote the space of all homogenous harmonic polynomials in $\R[x_1,\ldots,x_d]$ of degree $\ell$, and let $H_\ell$ denote the space of functions on $S^{d-1}$ obtained by restriction from $V_\ell$.  If $f \in H_\ell$, then we can write $f(r\theta) = g(\theta) r^{\ell}$ for a function $g \in H_\ell$, and setting \eqref{e.sphericallaplace} to zero at $r=1$ yields
$$\Delta_{S^{d-1}} g = -\ell(\ell + d - 2) g,$$
i.e., $g$ is an eigenfunction of $\Delta_{S^{n-1}}$ with eigenvalue $-\ell(\ell+d-2)$.  Note that \eqref{e.sphericallaplace} continues to be zero if we replace $\ell$ with the negative number $\ell' := -(d-2) - \ell$, since the expression $-\ell(\ell + d - 2)$ is unchanged by replacing $\ell$ with $\ell'$.  Thus, $g(\theta)r^{\ell'}$ is also harmonic on $\R^d \setminus \{0\}$.

Now, suppose that $g$ is normalized so that
	\begin{equation} \label{normalization}
	\int_{S^{d-1}} g(\theta)^2 d\theta = 1.
	\end{equation}
By scaling, the integral of $f^2$ over $\partial B_R(0)$ is thus given by $R^{d-1}R^{2\ell}$.  The $L^2$ norm on all of $B_R(0)$ is then given by
\begin{equation} \label{e.l2norm} \int_{B_R(0)} f(z)^2 dz = \int_0^R r^{d-1}r^{2\ell}dr = \frac{R^{d+2\ell}}{d+2\ell}.\end{equation}

A standard identity states that the Dirichlet energy of $g$, as a function on $S^{d-1}$, is given by the $L^2$ inner product $(-\Delta g, g) = \ell(\ell+d-2)$.  The square of $\| \nabla f \|$ is given by the square of its component along $S^{d-1}$ plus the square of its radial component.  We thus find that the Dirichlet energy of $f$ on $B_R(0)$ is given by
\begin{align*}\int_{B_R(0)} \| \nabla f(z)  \|^2 dz & = \ell(\ell+d-2) \int_0^R r^{d-1}r^{2(\ell-1)} dr   +   \int_0^R r^{d-1} r^{2(\ell-1)} \ell^2 dr \\ & = \frac{\ell(\ell+d-2)}{2\ell+d-2}R^{2\ell+d-2} +  \frac{\ell^2}{2\ell+d- 2} R^{2\ell + d-2} \\
& = \frac{2 \ell^2 + (d-2)\ell}{2\ell + (d-2)} R^{2\ell+d-2} \\ & = \ell  R^{2\ell+d-2}.
\end{align*}

Now suppose that we fix the value of $f$ on $\partial B_R(0)$ as above but harmonically extend it outside of $B_R(0)$ by writing $f(r \theta) = R^{\ell - \ell'} g(\theta)r^{\ell'}$ for $r > R$.  Then the Dirichlet energy of $f$ outside of $B_R(0)$ can be computed as

$$ R^{2(\ell - \ell')} \ell(\ell+d-2) \int_R^\infty r^{d-1}r^{2(\ell'-1)} dr +   R^{2(\ell - \ell')} \int_R^\infty r^{d-1} r^{2(\ell'-1)} (\ell')^2 dr,$$

which simplifies to

 \begin{align*} -\frac{\ell^2 + \ell(d-2) + (\ell')^2  }{2 \ell' + (d - 2)} R^{2\ell + d - 2} &= -\frac{\ell^2 + \ell(d-2) + \bigl(\ell + (d-2)\bigr)^2  }{2 (-\ell-(d-2)) + (d - 2)} R^{2\ell + d - 2} \\
 & = -\frac{2 \ell^2 + 3 \ell(d-2) + (d-2)^2}{- 2\ell - (d-2)} R^{2\ell + d - 2} \\
 &= (\ell + d-2)R^{2\ell + d - 2}.
 \end{align*}

Combining the inside and outside computations in the case $R=1$, we find that the harmonic extension $\tilde f$ of the function given by $g$ on $S^{d-1}$ has Dirichlet energy $2\ell + (d - 2)$.  If we decompose the GFF into an orthonormal basis that includes this $\tilde f$, we find that the component of $\tilde f$ is a centered Gaussian with variance $\frac{1}{2\ell + (d-2)}$.  If we replace $\tilde f$ with the harmonic extension of $g(R^{-1} \theta)$ (defined on $\partial B_R(0)$), then by scaling the corresponding variance becomes $\frac{1}{2\ell + (d-2)} R^{2-d}$.

Now in the augmented GFF the variance is instead given by \eqref{e.l2norm}, which amounts to replacing  $\frac{1}{2\ell + (d-2)}$ with $\frac{1}{2\ell + d}$.  Considering the component of  $g(R^{-1} \theta)$ in a basis expansion the space of functions on $\partial B_R(0)$ requires us to divide \eqref{e.l2norm} by $R^{2\ell}$ (to account for the scaling of $f$) and by $(R^{d-1})^2$ (to account for the larger integration area), so that we again obtain a variance of $\frac{1}{2\ell + d}R^{2-d}$ for the augmented GFF, versus $\frac{1}{2\ell + (d-2)} R^{2-d}$ for the GFF.

In some ways, the augmented GFF is very similar to the ordinary GFF: when we restrict attention to an origin-centered annulus, it is possible to construct independent Gaussian random distributions $h_1$, $h_2$, and $h_3$ such that $h_1$ has the law of a constant multiple of the GFF, $h_1+h_2$ has the law of the augmented GFF, and $h_1+h_2+h_3$ has the law of the ordinary GFF.

In light of Theorem \ref{t.fluctuations}, the following implies that (up to absolute continuity) the scaling limit of fixed-time $A_t$ fluctuations can be described by the GFF itself.

\begin{prop} \label{p.abscont}
When $d=2$, the law $\nu$ of the restriction of the GFF to the unit circle (modulo additive constant) is absolutely continuous w.r.t.\ the law $\mu$ of the restriction of the augmented GFF restricted to the unit circle.
\end{prop}

\begin{proof}
The relative
entropy of a Gaussian of density $e^{-x^2/2}$ with respect to a
Gaussian of density $\sigma^{-1} e^{-x^2/(2\sigma^2)}$ is given by
$$F(\sigma) = \int e^{-x^2/2} \left( (\sigma^{-2}-1)x^2/2 + \log \sigma \right) dx = (\sigma^{-2}-1)/2 + \log \sigma.$$
Note that $F'(\sigma) = -\sigma^{-3} + \sigma^{-1}$, and in particular $F'(1) = 0$.
Thus the relative entropy of a centered
Gaussian of variance $1$ with respect to a centered Gaussian of variance $1+a$ is $O(a^2)$.  This
implies that the relative entropy of $\mu$ with respect to $\nu$ --- restricted to the $j$th component $\alpha_j$ ---
is $O(j^{-2})$.  The same holds for the relative entropy of $\nu$ with respect to $\mu$.
Because the $\alpha_j$ are independent in both $\mu$ and $\nu$,
the relative entropy of one of $\mu$ and $\nu$ with respect to the other is the sum of
the relative entropies of the individual components, and this sum is finite.
\end{proof}

\section{General dimension}
\subsection{FKG inequality: Proof of Theorem \ref{t.fkg}}
We recall that increasing functions of a Poisson point process are non-negatively correlated \cite{GK}.  (This is easily derived from the more well known statement \cite{FKG} that increasing functions of independent Bernoulli random variables are non-negatively correlated.)  Let $\mu$ be the simple random walk probability measure on the space $\Omega'$ of walks $W$ beginning at the origin.  Then the randomness for internal DLA is given by a rate-one Poisson point process on $\mu \times \nu$ where $\nu$ is Lebesgue measure on $[0,\infty)$.  A realization of this process is a random collection of points in $\Omega' \times [0,\infty)$.  It is easy to see (for example, using the abelian property of internal DLA discovered by Diaconis and Fulton \cite{DF}) that adding an additional point $(w,s)$ increases the value of $A_{T(t)}$ for all times $t$.  The $A_{T(t)}$ are hence increasing functions of the Poisson point process, and are non-negatively correlated.  Since $F$ and $G$ are increasing functions of the $A_{T(t)}$, they are also increasing functions of the point process --- and are thus non-negatively correlated.

\subsection{Discrete harmonic polynomials}
\label{s.polynomials}

Let $\psi(x_1,\ldots,x_d)$ be a polynomial that is harmonic on $\R^d$, that is
	\[ \sum_{i=1}^d \frac{\partial^2 \psi}{\partial x_i^2} = 0. \]
In this section we give a recipe for constructing a polynomial $\psi_1$ that closely approximates $\psi$ and is discrete harmonic on $\Z^d$, that is,
	\[ \sum_{i=1}^d D_i^2 \psi_1 = 0  \]
where
	\[ D_i^2 \psi_1 = \psi_1(x+\basis_i) - 2 \psi_1(x) + \psi_1(x-\basis_i) \]
is the symmetric second difference in direction $\basis_i$.  The construction described below is nearly the same as the one given by Lov\'{a}sz in~\cite{Lovasz}, except that we have tweaked it in order to obtain a smaller error term: if $\psi$ has degree $k$, then $\psi - \psi_1$ has degree at most $k-2$ instead of $k-1$.
 Discrete harmonic polynomials have been studied classically, primarily in two variables: see for example Duffin \cite{Duffin}, who gives a construction based on discrete contour integration.

Consider the linear map
	\[ \Xi : \R[x_1,\ldots,x_d] \to \R[x_1,\ldots,x_d] \]
defined on monomials by
	 \[ \Xi ( x_1^{k_1} \cdots x_d^{k_d} ) = P_{k_1}(x_1) \cdots P_{k_d}(x_d) \]
where we define
	\[ P_{k}(x) = \prod_{j=-(k-1)/2}^{(k-1)/2} (x+j). \]

\begin{lemma}\label{discretepolynomial}
If $\psi \in \R[x_1,\ldots,x_d]$ is a polynomial of degree $k$ that is harmonic on $\R^d$, then the polynomial $\psi_1 = \Xi(\psi)$ is discrete harmonic on $\Z^d$, and $\psi - \psi_1$ is a polynomial of degree at most $k-2$.
\end{lemma}

\begin{proof}
An easy calculation shows that
	\[ D^2 P_k = k(k-1) P_{k-2} \]
from which we see that
	\[ D_i^2 \Xi [\psi] = \Xi [ \frac{\partial^2}{\partial x_i^2} \psi ]. \]
If $\psi$ is harmonic, then the right side vanishes when summed over $i=1,\ldots,d$, which shows that $\Xi [\psi]$ is discrete harmonic.

Note that $P_k(x)$ is even for $k$ even and odd for $k$ odd.  In particular, $P_k(x) - x^k$
has degree at most $k-2$, which implies that $\psi -\psi_1$ has degree at most $k-2$.
\end{proof}

To obtain a discrete harmonic polynomial $\psi_{(m)}$ on the lattice $\frac{1}{m} \Z^d$, we set
	\[ \psi_{(m)}(x) := m^{-k} \psi_1(mx), \]
where $k$ is the degree of $\psi$.

\subsection{General-dimensional CLT: Proof of Theorem \ref{t.highdconvergence}}

\proofof{Theorem \ref{t.highdconvergence}}
For each fixed $\psi$ with $\psi(0)=0$ the process $M(t) := \Phi_A^m(\psi,t)$ is a martingale in $t$.  Each time a new particle is added, we can imagine that it performs Brownian motion on the grid (instead of a simple random walk), which turns $M$ into a continuous martingale, as in \cite{JLS}.  By the martingale representation theorem (see~\cite[Theorem~V.1.6]{RY}), we can write $M(t) = \beta(s_m(t))$, where $\beta$ is a standard Brownian motion and 
	\[ s_m(t) := \limsup_{\substack{0=t_0<t_1<\cdots<t_{n}=t \\ |t_{i+1}-t_i|\to 0}} \; \sum_{i=0}^{n-1} (M(t_{i+1}) - M(t_i))^2 \]
is the quadratic variation of $M$ on the interval $[0,t]$.  To show that $\Phi_A^m(\psi,t)$ converges in law as $m \to \infty$ to a Gaussian with variance $V:= \int_{B_{r(t)}(0)} \psi(x)^2dx$, it suffices to show that for fixed~$t$ the random variable $s_m(t)$ converges in law to~$V$.

By standard Riemann integration and the $A_t$ fluctuation bounds in \cite{JLS, JLS2} (the weaker bounds of \cite{LBG} would also suffice here) we know that $m^{-d} \sum_{x \in A_{tm^d}} \psi_{(m)}(x/m)^2$ converges in law to $\int_{B_{r(t)}} \psi(x)^2dx$ as $m \to \infty$.  Thus it suffices to show that \begin{equation} \label{e.smt} s_m(t) - m^{-d} \sum_{x \in A_{tm^d}} \psi_{(m)}(x/m)^2\end{equation} converges in law to zero.  This expression is actually a martingale in $t$.  Its expected square is the sum of the expectations of the squares of its increments, each of which is $O(m^{-2d})$.  The overall expected square of \eqref{e.smt} is thus $O(m^{-d})$, which indeed tends to zero.

Recall \eqref{e.phia} and note that if we replace $t$ with $T(t)$, this does not change the convergence in law of $s_m(t)$ when $\psi(0)=0$.  However, when $\psi(0) \not = 0$, it introduces an asymptotically independent source or randomness which scales to a Gaussian of variance $\psi(0)^2 t$ (simply by the central limit theorem for the Poisson point process), and hence \eqref{e.var} remains correct in this case.

Similarly, suppose we are given $0 = t_0< t_1 < t_2 < \ldots < t_\ell$ and functions $\psi_1, \psi_2, \ldots \psi_\ell$.  The same argument as above, using the martingale
in $t$,
$$
m^{-d/2} \sum_{j=1}^\ell \sum_{A_{T(m^d(t\wedge t_j))}} \psi_{j,(m)}(x/m) 
- t\psi_{j,(m)}(0),
$$
implies that $\sum_{i=1}^\ell \Phi_A^m(\psi_j,t_j)$ converges in law to a Gaussian with variance
$$\sum_{j=1}^{\ell} \int_{B_{r(t_j)} \setminus B_{r(t_{j-1})}} \left( \sum_{i=j}^\ell \psi_i(x) \right)^2dx.$$
The theorem now follows from a standard fact about Gaussian random variables on a finite dimensional vector spaces (proved using characteristic functions): namely, a sequence of random variables on a vector space converges in law to a multivariate Gaussian if and only if all of the one-dimensional projections converge.  The law of~$h$ is determined by the fact that it is a centered Gaussian with covariance given by \eqref{e.cov}.
\qed


\section{Dimension two}

\subsection{Two dimensional CLT: Proof of Theorem~\ref{t.lateness}}

Recall that $A_t$ for $t \in \Z_{+}$ denotes the discrete-time IDLA cluster with exactly $t$ sites, and $A_T = A_{T(t)}$ for $t \in \R_{+}$ denotes the continuous-time cluster whose cardinality is Poisson-distributed with mean $t$.

Define
\[
F_0(t) : = \inf \{t: z\in A_t\}
\]
and
\[
L_0(z) := \sqrt{F_0(z)/\pi} - |z|.
\]
Fix $N<\infty$, and consider a test function of the form
\[
\f(re^{i\theta}) = \sum_{|k| \le N} a_k(r) e^{ik\theta}
\]
where the $a_k$ are smooth functions supported in an interval $0 < r_0 \le r \le r_1 <\infty$.
We will assume, furthermore, that $\f$ is real-valued.  That is, the complex numbers $a_k$
satisfy
\[
a_{-k}(r) = \overline{a_k(r)}
\]

\begin{theorem}\label{gffdisk0}  As $R\to \infty$,
\[
\frac{1}{R^2} \sum_{z\in (\Z + i\Z)/R}  L_0(Rz) \frac{\phi(z)}{|z|^2} \longrightarrow N(0,V_0)
\]
in law, where
\[
V_0 =  \sum_{0 < |k| \le N}
\ 2\pi \int_0^\infty \left| \int_\rho^\infty  a_k(r) (\rho/r)^{|k|+1} \frac{dr}{r}\right|^2 \frac{d\rho}{\rho}.
\]
\end{theorem}

It follows from Lemma \ref{dual} below (with $q = |k| + 1$, $y = \log r$), that Theorem~\ref{gffdisk0}
can be interpreted as saying
that $L_0(Rz)$ tends weakly to the Gaussian random
distribution associated to the Hilbert space $H_{nr}^1$ with norm
\[
\|\eta\|_0^2 = \sum_{0 < |k| <\infty} 2\pi \int_0^\infty [|r\de_r\eta_k|^2 + (|k|+1)^2|\eta_k|^2]\frac{dr}{r}
\]
where
\[
\eta_k(r) = \frac{1}{2\pi}\int_0^{2\pi}  \eta(r e^{i\theta} )e^{-ik\theta}d\theta
\]
and $\eta_0(r) \equiv 0$.  (The subscript $nr$ means nonradial:  $H_{nr}^1$ is the orthogonal complement of radial functions in the Sobolev space $H^1$.)

\begin{lemma}\label{dual} Let $q\ge0$ and let $\psi$ be a real valued function on $\R$.  Denote
\[
\|\psi\|_q = \sup \int_{-\infty}^\infty \psi(y) f(y) dy
\]
where the supremum is over real-valued $f$,  compactly supported and subject to the constraint
\[
\int_{-\infty}^\infty (f'(y)^2 + q^2 f(y)^2)dy \le 1.
\]
Then
\[
\|\psi\|_q^2 =  \int_{-\infty}^\infty \left|\int_s^\infty \psi(y)e^{q(s-y)}dy\right|^2 ds.
\]
\end{lemma}

\begin{proof}
In the case $q=0$, replace $f$ in $\displaystyle \int \psi f \, dy$ with
\[
f(y) = \int_{-\infty}^y f'(s)ds
\]
change order of integration and apply the Cauchy-Schwarz inequality.
For the case $q>0$, multiply by the appropriate factors $e^{qy}$ and
$e^{2qy}$ to deduce this from the case $q=0$.
\end{proof}

If we use $A_T$ and corresponding functions $F(z)$ and $L(z)$, then the $a_0$ coefficient
figures in the limit formula as follows.
\begin{theorem}\label{gffdisk1}  As $R\to \infty$,
\[
\frac{1}{R^2} \sum_{z\in (\Z + i\Z)/R}  L(Rz) \frac{\phi(z)}{|z|^2} \longrightarrow N(0,V)
\]
in law, where
\[
V =  \sum_{|k| \le N} 2\pi \int_0^\infty \left| \int_\rho^\infty a_k(r) (\rho/r)^{|k|+1} \frac{dr}{r}\right|^2 \frac{d\rho}{\rho}.
\]
\end{theorem}
Theorem~\ref{gffdisk1} is a restatement of Theorem \ref{t.lateness}.  (As in the proof
of Theorem~\ref{t.highdconvergence}, the convergence in law of all one-dimensional
projections to the appropriate normal random variables implies the corresponding result for the joint distribution of any finite collection
of such projections.)   Theorem~\ref{gffdisk1}  says that
$L(Rz)$ tends weakly to a Gaussian distribution for the Hilbert space $H^1$ with the
norm
\[
\|\eta\|^2 = \sum_{k=-\infty}^\infty  2\pi \int_0^\infty [|r\de_r\eta_k|^2 + (|k|+1)^2|\eta_k|^2]\frac{dr}{r}.
\]
By way of comparison, the usual Gaussian free field is the one associated to the Dirichlet norm
\[
\int_{\R^2} |\nabla \eta|^2 dxdy = \sum_{k=-\infty}^\infty  2\pi
\int_0^\infty [|r\de_r\eta_k|^2 + k^2 |\eta_k|^2]\frac{dr}{r}.
\]
Comparing these two norms, we see that the second term in $\| \eta \|^2$ has an additional $+1$, hence our choice of the term ``augmented Gaussian free field.''  As derived in \textsection\ref{ss.augmentedGFF}, this $+1$ results in a smaller variance $\frac{1}{2\ell + d} R^{2-d}$ in each spherical mode of degree $\ell$ of the augmented GFF, as compared to $\frac{1}{2\ell+d-2} R^{2-d}$ for the usual GFF.  The surface area of the sphere is implicit in the normalization \eqref{normalization}, and is accounted for here in the factors
$2\pi$ above.

To prove Theorem~\ref{gffdisk0}, write
\begin{align*}
L_0(z) &=
\frac{1}{2\sqrt{\pi}}\int_0^\infty(1-1_{A_t})t^{1/2}\frac{dt}{t}
- \frac{1}{2\sqrt{\pi}}\int_0^\infty(1-1_{\pi|z|^2 \le t})t^{1/2}\frac{dt}{t} \\
& = \frac{1}{2\sqrt{\pi}}\int_0^\infty(1_{\pi|z|^2 \le t}-1_{A_t})t^{1/2}\frac{dt}{t}.
\end{align*}

Let $p_0(z)=1$, and for $k\geq 1$ let $p_k(z) = q_k(z)-q_k(0)$, where
	\[ q_k(z) = \Xi [z^k] \]
is the discrete harmonic polynomial associated to $z^k = (x + iy)^k$ as described in \textsection\ref{s.polynomials}. The sequence $p_k$ begins
	\[ 1, \;  z,  \; z^2, \; z^3 - \frac14 \bar{z}, \; z^4 - z \bar{z}, \; \ldots . \]
 For instance, to compute $p_3$, we expand
	\[ z^3 = x^3 - 3xy^2 + i \left[ 3x^2y - y^3 \right] \]
and apply $\Xi$ to each monomial, obtaining
	\[ p_3(z) = (x-1)x(x+1) - 3x(y-\frac12)(y+\frac12) + i \left[ 3(x-\frac12)(x+\frac12)y - (y-1)y(y+1) \right] \]
which simplifies to $z^3 - \frac14 \bar z$.  One readily checks that this defines a discrete harmonic function on $\Z + i\Z$. (In fact, $z^3$ is itself discrete harmonic, but $z^k$ is not for $k\geq 4$.)
To define $p_k$ for negative $k$, we set $p_{-k} (z) = \overline{p_k(z)}$.

Define
\[
\psi(z, t,R) = \sum_{k=-N}^N a_k(\sqrt{t/\pi R^2}) p_k(z)(\sqrt{t/\pi})^{-|k|}
\]
and
\[
\psi_0(z,t,R) = \psi(z,t,R) - a_0(\sqrt{t/\pi R^2})
\]
\begin{lemma}\label{discrete-error}  If $c_1 R^2 \le t \le c_2R^2$ and
$| |z| - \sqrt{t/\pi}| \le C\log R $, then
\[
\left|  \psi(z,t,R)- \phi(z/R) \right|
\le C (\log R) /R
\]
\end{lemma}
This lemma follows easily from the fact that the coefficients $a_k$ are smooth
and the bound $|p_k(z) - z^k| \le C|z|^{|k|-1}$.

\subsection{Van der Corput bounds}

\begin{lemma}\label{vandercorput}  (Van der Corput)
\begin{itemize}
\item[(a)] $|\# \{z \in \Z + i\Z: \pi |z|^2 \le t\} - t| \le C t^{1/3}$.
\item[(b)] For $k \ge 1$,
\[
t^{-k/2} \left| \sum_{ z\in \Z + i\Z} z^k \, 1_{\pi |z|^2 \le t}  \right| \le C t^{1/3}.
\]
\item[(c)] For $k \ge 1$,
\[
t^{-k/2} \left| \sum_{ z\in \Z + i\Z} p_k(z) \, 1_{\pi|z|^2 \le t}  \right| \le C t^{1/3}.
\]
\end{itemize}
\end{lemma}

Part (a) of this lemma was proved by van der Corput in the 1920s (See \cite{GS10},
Theorem 87 p.~484).  Part (b) follows from the
same method, as proved below.  Part (c) follows from part (b) and
the stronger estimate of Lemma~\ref{discretepolynomial},
$|p_k(z) - z^k| \le C|z|^{k-2}$ for $k\ge2$ (and $p_1(z) -z = 0$).

We prove part (b) in all dimensions.  Let $P_k$ be a harmonic polynomial on $\R^d$
of homogeneous of degree $k$.
Normalize so that
\[
\max_{x\in B} |P_k(x)| =1
\]
where $B$ is the unit ball.  In this discussion $k$ will be fixed
and the constants are allowed to depend on $k$.

We are going to show that for $k\ge1$,
\[
\left|  \frac{1}{R^d} \sum_{|x|<R, \, x\in \Z^d} P_k(x)/R^k \right| \le R^{-1-\alpha}
\]
where
\[
\alpha = 1 - 2/(d+1)
\]
For $d=2$, $\alpha = 1/3$, and $R^d R^{-1-\alpha} = R^{2/3} \approx t^{1/3}$.  This
is the claim of part (b).


The van der Corput theorem is the case $k=0$.  It says
\[
(1/R^d)\abs{\#\set{x\in \Z^d: |x|<R} - \vol(|x|<R)}  \le R^{-1-\alpha}
\]
Let $\e = 1/R^\alpha$.


Consider $\rho$ a smooth, radial function on $\R^d$ with integral $1$
supported in the unit ball.   Then define $\chi = 1_B$
characteristic function of the unit ball.
Denote
\[
\rho_\e(x) = \e^{-d}\rho(x/\e), \quad
\chi_R(x) = R^{-d} \chi(x/R)
\]
Then
\[
\left|  \sum_{x\in \Z^d} (\chi_R*\rho_\e(x)-\chi_R(x))P_k(x)/R^k \right|
\le R^{-1-\alpha}
\]
This is because $\chi_R*\rho_\e(x) - \chi_R(x)$ is nonzero
only in the annulus of width $2\e$ around $|x|=R$ in which
(by the van der Corput bound) there are $O(R^{d-1}\e)$ lattice points.

The Poisson summation formula implies
\[
\sum_{x\in \Z^d} \chi_R*\rho_\e(x)P_k(x)/R^k
=
\sum_{\xi\in 2\pi\Z^d}[ \hat \chi_R(\xi) \hat\rho_\e(\xi)]*\hat P_k(\xi)/R^k
\]
in the sense of distributions.  The Fourier transform of a polynomial
is a derivative of the delta function,
$\hat P_k(\xi) = P_k(i\de_\xi)\delta(\xi)$.
Because $k\ge1$ and $P_k(x)$ is harmonic,
its average with repect to any radial function is zero. This is
expressed in the dual variable as the fact that when
$\xi=0$,
\[
P_k(i\de_\xi) (\hat \chi_R(\xi) \hat\rho_\e (\xi)]) = 0
\]
So we our sum equals
\[
\sum_{\xi\neq0, \,\xi\in 2\pi\Z^d}[ \hat\chi_R(\xi) \hat \rho_\e(\xi)]*\hat P_k (\xi)/R^k
\]
Next look at
\[
\hat \chi_R(\xi) = \hat\chi(R\xi)
\]
\[
P_k(i\de_\xi)\hat \chi(R\xi) = R^k \int_{|x|< 1} P_k(x)e^{-iRx\cdot\xi} dx
\]
All the terms in which fewer derivatives fall on $\hat\chi_R$ and more
fall on $\rho_\e$ give much smaller expressions:  the factor $R$ corresponding
to each such differentiation is replaced by an $\e$.

The asymptotics of this oscillatory integral above are well
known.  For any fixed polynomial $P$ they are of the same
order of magnitude as for $P\equiv 1$, namely
\[
|P_k(i\de_\xi)\hat\chi(R\xi)|/R^k \le C_k |R\xi|^{-(d+1)/2}
\]
This is proved by the method of stationary phase and can also
be derived from well known asymptotics of Bessel functions.

It follows that our sum is majorized by (replacing the letter $d$ by $n$ so that
it does not get mixed up with the differential $dr$)
\begin{align*}
\int_1^{\infty} (Rr)^{-(n+1)} \frac{r^{n-1}dr}{(1 + \e r)^N}
& \approx
\int_1^{1/\e} (Rr)^{-(n+1)} r^n \frac{dr}{r}\\
& \approx  R^{-(n+1)/2} \e^{-(n-1)/2} \\
&  = R^{-1 - \alpha}.
\end{align*}

\subsection{The main part of the proof}

Denote
\[
X_R = \frac{1}{R^2} \sum_{z\in (\Z + i\Z)/R}  L_0(Rz) \frac{\phi(z)}{|z|^2}.
\]
Applying the formula above for $L_0$,
\begin{align*}
X_R
& =
\sum_{z\in \Z + i\Z}  L_0(z) \frac{\phi(z/R)}{|z|^2}  \\
& =
\frac{1}{2\sqrt\pi} \int_0^\infty   \sum_{z\in \Z+i\Z} (1_{\pi |z|^2 \le t} - 1_{A_t})
\frac{\phi(z/R)}{|z|^2} t^{1/2}\frac{dt}{t} \\
& =
\frac{1}{2\sqrt{\pi}} \int_0^\infty   \sum_{z\in \Z+i\Z} (1_{\pi |z|^2 \le t} - 1_{A_t})
\frac{\psi(z,t, R)}{t/\pi} t^{1/2}\frac{dt}{t}  + E_R.
\end{align*}
To estimate the error term $E_R$,   note first that
the coefficients $a_k$ are supported in a fixed annulus,
the integrand above is supported in the range $c_1 R^2 \le  t \le c_2 R^2$.
Furthermore, by \cite{JLS}, there is an absolute constant $C$ such that for all sufficiently large $R$ and all $t$ in this range, the difference $\displaystyle  1_{\pi |z|^2 \le t} - 1_{A_t}$ is supported on the set of $z\in \Z^2$ such that $| |z| - \sqrt{t/\pi}| \le C \log R$.  Thus
\[
 \sum_{z\in \Z + i\Z} |1_{\pi |z|^2 \le t} - 1_{A_t}| \le K R \log R.
\]
Moreover,  Lemma \ref{discrete-error} applies and
\[
|E_R|  \le  C\int_{c_1R^2}^{c_2R^2} (R\log R) \frac{\log R}{R} t^{-1/2} \frac{dt}{t} = O((\log R)/R).
\]
Next, Lemma \ref{vandercorput}(a) says (since $\# A_t = t$)
\[
\left| \sum_{z\in \Z + i\Z} 1_{\pi |z|^2 \le t} - 1_{A_t}\right| \le  C t^{1/3}.
\]
Thus replacing $\psi$ by $\psi_0$ gives an additional error of size at most
\[
C\int_{c_1R^2}^{c_2R^2} t^{1/3} t^{-1/2} \frac{dt}{t} = O(R^{-1/3}).
\]
In all,
\begin{equation}\label{XRdiscrete}
X_R =
\frac{\sqrt\pi}{2} \int_0^\infty   \sum_{z\in \Z+i\Z} (1_{\pi |z|^2 \le t} - 1_{A_t})
\psi_0(z,t, R)  t^{-1/2}\frac{dt}{t}  + O(R^{-1/3})
\end{equation}

For $s = 0, 1, \dots$, consider the process
\[
M(s) = \frac{\sqrt\pi}{2} \int_0^\infty   \sum_{z\in \Z+i\Z} (1_{\pi |z|^2 \le t} - 1_{A_{s\wedge t}})
\psi_0(z,t, R) t^{-1/2}\frac{dt}{t}
\]
Note that $M(s)\to X_R$ as $s \to \infty$.   Note also that
Lemma \ref{vandercorput}(c) implies
\[
M(0) = O(R^{-1/3}).
\]
Because $p_k$ are discrete harmonic and $p_k(0)  =0$ for all $k\neq0$,
$M(s) - M(0)$ is a martingale.  It remains to show that
$M(s)-M(0) \longrightarrow N(0,V_0)$ in law.  As outlined below, this will follow from the martingale
central limit theorem (see, e.g., \cite{Brown,HH} or \cite[p.~414]{Durrett}).

For sufficiently large~$R$, the difference $M(s+1) - M(s)$ is nonzero only for $s$ in the range $c_1 R^2 \le s \le c_2 R^2$; and $|F_0(z) - \pi|z|^2| \le C R \log R$.  We now show that this implies
\begin{equation}\label{oneincrement}
|M(s+1) - M(s)|^2 = O(1/R^2)
\end{equation}
and
\begin{equation}\label{quad}
\sum_{s=0}^\infty |M(s+1) - M(s)|^2 = V_0 + O((\log R)/R)
\end{equation}
so that the martingale central limit theorem applies.

To prove \eqref{oneincrement}, observe that
\[
M(s+1)-M(s) =
-\frac{\sqrt\pi}{2} \int_{F_0(z)}^\infty \psi_0(z,t,R) t^{-1/2}\frac{dt}{t}
\]
where $z$ is the $(s+1)$th point of $A_t$.
Then
$|z| \le \sqrt{t/\pi} + K\log R$ implies $|p_k(z)| (t/\pi)^{-|k|/2} \le  C$,
and hence
\[
|\psi_0(z,t,R)| \le C
\]
Recalling that $\psi_0= 0$ unless $c_1R^2 \le t \le c_2 R^2$,
we have
\[
|M(s+1)-M(s)| \le C \int_{c_1R^2}^{c_2R^2}   t^{-1/2} \frac{dt}{t} = O(1/R)
\]
which confirms \eqref{oneincrement}.

Because $A_t$ fills the lattice $\Z + i\Z$ as $t\to \infty$, we have
\begin{align*}
\sum_{s=0}^\infty  & |M(s+1) - M(s)|^2   \\
&=
\sum_{z\in \Z+i\Z}
\left|
\frac{\sqrt\pi}{2} \int_{F_0(z)}^\infty \sum_{0 < |k| \le N}  a_k(\sqrt{t/\pi R^2}) p_k(z)(t/\pi)^{-|k|/2} t^{-1/2}\frac{dt}{t}
\right|^2.
\end{align*}
We prove \eqref{quad} in three steps:  replace $p_k(z)$ by $z^k$ (or $\bar z^{|k|}$ if $k<0$);
replace the lower limit $F_0(z)$ by $\pi |z|^2$; replace the sum of $z$ over lattice sites
with the integral with respect to Lebesgue measure in the complex $z$-plane.

We begin the proof of \eqref{quad} by noting that
the error term introduced by replacing $p_k$ with $z^k$ is
\[
|p_k(z) - z^k| (t/\pi)^{-|k|/2} \le  C_kt^{-1}  = O(1/R^2)
\]
In the integral this is majorized by
\[
\int_{c_1R^2}^{c_2R^2}   t^{-1/2} \frac{dt}{t}
\int_{c_1R^2}^{c_2R^2}   \frac{1}{R^2} t^{-1/2} \frac{dt}{t} = O(1/R^4)
\]
Since there are $O(R^2)$ such terms, this change contributes order
$R^2/R^4 = 1/R^2$ to the sum.

Next,  we change the lower limit from $F_0(z)$ to $\pi|z|^2$.
Since $|F_0(z) -\pi |z|^2| \le C R \log R$, the integral inside
$|\cdots|^2$ is changed by
\[
\int_{F_0(z)}^{\pi|z|^2}  1_{c_1 R^2 \le c_2 R^2}  t^{-1/2} \frac{dt}{t} = O((\log R)/R^2)
\]
Thus the change in the whole expression is majorized by the order of the cross term
\[
(1/R) (\log R)/R^2 = (\log R)/R^3
\]
Again there are $R^2$ terms in the sum over $z$, so the sum of the errors
is $O((\log R)/R)$.

Lastly, we replace the value at each site $z_0$ by the integral
\[
\int_{Q_{z_0}}
\left|
\frac{\sqrt\pi}{2} \int_{\pi r^2} ^\infty \sum_{0 < |k| \le N}  a_k(\sqrt{t/\pi R^2}) r^ke^{ik\theta} (t/\pi)^{-|k|/2} t^{-1/2}\frac{dt}{t}
\right|^2rdrd\theta
\]
where $Q_{z_0}$ is the unit square centered at $z_0$ and $z = r e^{i\theta}$.
Because the square has area $1$, the term in the lattice sum is the same as
this integral with $z=re^{i\theta}$ replaced by $z_0$ at each occurrence.
Since $|z-z_0| \le \sqrt2$,
\[
|z^k - z_0^k| \le 4k (|z| + |z_0|)^{k-1}   = O(R^{k-1})
\]
After we divide by $(\sqrt{t/\pi})^k$, the order of error is $1/R$.  Adding all
the errors contributes at most order  $1/R$ to the sum.  We must also
take into account the change in the lower limit of the integral, $\pi|z_0|^2$ is replaced
by $\pi|z|^2 = \pi r^2$.  Since $|z-z_0| \le \sqrt2$,
\[
||z|^2 - |z_0|^2| \le \sqrt2 (|z| + |z_0|)  \le CR
\]
Recall that in the previous step we previously changed the lower limit by $O(R\log R)$.
Thus by the same argument, this smaller change gives rise to an error of order $1/R$ in
the sum over $z_0$.

The proof of \eqref{quad} is now reduced to evaluating
\[
\int_0^{2\pi}\int_0^\infty
\left|
\frac{\sqrt\pi}{2} \int_{\pi r^2} ^\infty \sum_{0 < |k| \le N}  a_k(\sqrt{t/\pi R^2}) r^{|k|}e^{ik\theta} (t/\pi)^{-|k|/2} t^{-1/2}\frac{dt}{t}
\right|^2rdrd\theta
\]
Integrating in $\theta$ and changing variables from $r$ to $\rho = r/R$,
\[
= \frac{\pi^2}{2}
\sum_{0< |k| \le N}
\int_0^\infty \left|
\int_{\pi \rho^2R^2} ^\infty
a_k(\sqrt{t/\pi R^2}) (R\rho)^{|k|+1}  (t/\pi)^{-|k|/2} t^{-1/2}\frac{dt}{t}
\right|^2\frac{d\rho}{\rho}
\]
Then change variables from $t$ to to $r = \sqrt{t/\pi R^2}$ to obtain
\[
= 2\pi
\sum_{0< |k| \le N}
\int_0^\infty \left|
\int_{\rho} ^\infty
a_k(r)   (\rho/r)^{|k | + 1} \frac{dr}{r}
\right|^2\frac{d\rho}{\rho} = V_0.
\]
This ends the proof of Theorem \ref{gffdisk0}.

The proof of Theorem \ref{gffdisk1} follows the same idea.  We replace $A_t$ by
the Poisson time region $A_{T}$ (for $T = T(t)$), and we need to
find the limit as $R\to \infty$ of
\begin{align*}
\frac{\sqrt\pi}{2} \int_0^\infty   & (t - \#A_T) a_0(\sqrt{t/\pi R^2})  t^{-1/2}\frac{dt}{t}  \\
& +
\frac{\sqrt\pi}{2} \int_0^\infty   \sum_{z\in \Z+i\Z} (1_{\pi |z|^2 \le t} - 1_{A_T})
\psi_0(z,t, R)  t^{-1/2}\frac{dt}{t}
\end{align*}
The error terms in the estimation showing this quantity is within $O(R^{-1/3})$ of
\[
\frac{1}{R^2} \sum_{z\in (\Z + i\Z)/R}  L(Rz) \frac{\phi(z)}{|z|^2}
\]
are nearly the same as in the previous proof.  We describe briefly
the differences.  The difference between Poisson time and ordinary
counting is
\[
|\# A_T - \# A_t| =
|\# A_T - t| \le C t^{1/2} \log t = O(R \log R)  \quad \mbox{almost surely}
\]
if $t \approx R^2$.   It follows that  for $|z| \approx R$,
\[
|F(z) - \pi |z|^2| = O( R \log R) \quad \mbox{almost surely}
\]
as in the previous proof for $F_0(z)$.   Further errors are
also controlled since we then have the estimate analogous to the
one above for $A_t$, namely
\[
 \sum_{z\in \Z + i\Z} |1_{\pi |z|^2 \le t} - 1_{A_T}| \le  C R \log R
\]
We consider the continuous time martingale
\begin{align*}
M(s) &=
\frac{\sqrt\pi}{2} \int_0^\infty   (s\wedge t - \# A_{T(s \wedge t)}) a_0(\sqrt{t/\pi R^2})  t^{-1/2}\frac{dt}{t}  \\
& \quad +
\frac{\sqrt\pi}{2} \int_0^\infty   \sum_{z\in \Z+i\Z} (1_{\pi |z|^2 \le t} - 1_{A_{T(s\wedge t)}})
\psi_0(z,t, R)  t^{-1/2}\frac{dt}{t}
\end{align*}
Instead of using the martingale central limit theorem, we use the martingale
representation theorem.  This says
that the martingale $M(s)$ when reparameterized by its quadratic variation has the
same law as Brownian motion.   We must show that almost
surely the quadratic variation of $M$  on $0 \le s < \infty$ is $V + O(R^{-1/3})$.
\begin{align*}
\lim_{\epsilon \to 0}&  \EE ((M(s+\epsilon)- M(s))^2 | A_{T(s)})/\epsilon  \\
& =
\frac{1}{2\pi}
\int_0^{2\pi} \left|
\frac{\sqrt\pi}{2} \int_{s} ^\infty \sum_{ |k| \le N}  a_k(\sqrt{t/\pi R^2}) e^{ik\theta}) (s/t)^{|k|/2} t^{-1/2}\frac{dt}{t}
\right|^2d\theta \\
& \qquad \qquad + O(R^{-1/3})
\end{align*}
Integrating with respect to $s$ gives the quadratic variation $V + O(R^{-1/3})$ after
a suitable  change of variable as in the previous proof.
	
\subsection{Fixed time fluctuations: Proof of Theorem \ref{t.fluctuations}} \label{ss.fixedtimeproof}

Theorem \ref{t.fluctuations} follows almost immediately from the $d=2$ case of Theorem \ref{t.highdconvergence} and the estimates above.  Consider $(\phi, \tilde E_t)$ where $\tilde E_t$ is as in \eqref{e.et2}.  What happens if we replace $\phi$ with a function $\tilde \phi$ that is discrete harmonic on the rescaled mesh $m^{-1} \Z^d$ within a $\log m/m$ neighborhood of $B_1(0)$?  Clearly, if $\phi$ is smooth, we will have $\phi- \tilde \phi = O( m^{-1} \log m)$.  Since there are at most $O(m^{d-1} \log m)$ non-zero terms in \eqref{e.et2}, the discrepancy in \begin{equation} \label{e.Oet} (\phi, \tilde E_t) - (\tilde \phi, \tilde E_t)=  O\bigl( m^{-d/2} m^{d-1} (m^{-1} \log m) \log m\bigr) = O \bigl( m^{d/2-2} (\log m)^2 \bigr), \end{equation} which tends to zero as long as $d \in \{2,3\}$.

The fact that replacing $E_t$ with $\tilde E_t$ has a negligible effect follows from the above estimates when $d=2$.  This may also hold when $d=3$, but we will not prove it here.  Instead we remark that Theorem \ref{t.fluctuations} holds in three dimensions provided that we replace \eqref{e.et} with \eqref{e.et2}, and that the theorem as stated probably fails in higher dimensions even if we make a such a replacement.  The reason is that \eqref{e.et2} is positive at points slightly outside of $\B_r$ (or outside of the support of $w_t$) and negative at points slightly inside.  If we replace a discrete harmonic polynomial $\psi$ with a function that agrees with $\psi$ on $B_1(0)$ but has a different derivative along portions of $\partial B_1(0)$, this may produce a non-trivial effect (by the discussion above) when $d \geq 4$.

Finally, we note that replacing $\psi_m$ by $\psi$ introduces an error of order $m^{-2}$, and the same argument as above gives
\begin{equation} (\psi, \tilde E_t) - (\tilde \psi_m, \tilde E_t)=  O\bigl( m^{-d/2} m^{d-1} m^{-2} \log m\bigr) = O \bigl( m^{d/2-3} (\log m) \bigr), \end{equation}
which tends to zero when $d \in \{2,3,4,5\}$.

\end{document}